\documentclass{amsart}

\usepackage{graphicx}

\usepackage{hyperref}

\numberwithin{equation}{section}
\numberwithin{figure}{section}

\allowdisplaybreaks

\DeclareSymbolFont{bbold}{U}{bbold}{m}{n}
\DeclareSymbolFontAlphabet{\mathbbold}{bbold}
\newcommand{\ind}{\mathbbold{1}}

\theoremstyle{plain} \newtheorem{theorem}{Theorem}[section]
\theoremstyle{plain} 
\theoremstyle{plain} 
\theoremstyle{plain} \newtheorem{proposition}[theorem]{Proposition}
\theoremstyle{remark} \newtheorem{remark}[theorem]{Remark}
\theoremstyle{definition} \newtheorem{example}[theorem]{Example}
\theoremstyle{definition} 
\theoremstyle{remark}

\begin{document}

\title{Unseparated pairs and fixed points in random permutations}

\author{Persi Diaconis},
\address{Persi Diaconis \\
Department of Mathematics \\
Stanford University \\
Building 380, Sloan Hall \\
Stanford, CA 94305 \\
U.S.A.}
\email{diaconis@math.stanford.edu}
\urladdr{http://www-stat.stanford.edu/~cgates/PERSI/}

\author{Steven N.\ Evans}
\address{Steven N.\ Evans \\
  Department of Statistics \#3860 \\
  University of California at Berkeley \\
  367 Evans Hall \\
  Berkeley, CA 94720-3860 \\
  U.S.A.}
\email{evans@stat.Berkeley.EDU}
\urladdr{http://www.stat.berkeley.edu/users/evans/}

\author{Ron Graham}
\address{Ron Graham \\
Department of Computer Science \\
University of California at San Diego \\
9500 Gilman Drive \#0404 \\
La Jolla, CA  92093-0404 \\
U.S.A.}
\email{graham@ucsd.edu}
\urladdr{http://cseweb.ucsd.edu/~rgraham/}

\thanks{PD supported in part by NSF grant DMS-08-04324, 
SNE supported in part by NSF grant DMS-09-07630,
RG supported in part by NSF grant DUE-10-20548}


\begin{abstract}
In a uniform random permutation $\Pi$ of $[n] := \{1,2,\ldots,n\}$,
the set of elements $k \in [n-1]$ such that $\Pi(k+1) = \Pi(k) + 1$ has
the same distribution as the set of fixed points of $\Pi$ that lie in $[n-1]$.
We give three different proofs of this fact using, respectively, an
enumeration relying on the inclusion-exclusion principle,
the introduction of two different Markov chains to generate uniform
random permutations, and the construction of a combinatorial bijection.
We also obtain the distribution of the analogous set 
for circular permutations that consists of those $k \in [n]$ such that
$\Pi(k+1 \mod n) = \Pi(k) + 1 \mod n$.  This latter random set is just
the set of fixed points of the commutator
$[\rho, \Pi]$, where $\rho$ is the $n$-cycle $(1,2,\ldots,n)$.  We show
for a general permutation $\eta$ that, under weak conditions
on the number of fixed points and $2$-cycles of $\eta$, the 
total variation distance between the distribution of the number of fixed
points of $[\eta,\Pi]$ and a Poisson distribution with expected value $1$
is small when $n$ is large.
\end{abstract}

\maketitle

\section{Introduction}

The goal of any procedure for shuffling a deck of $n$ cards
labeled with, say, $[n] := \{1,2,\ldots,n\}$ is to take
the cards in some specified original order, which we may
take to be $(1,2,\ldots,n)$,  and re-arrange them randomly
in such a way that all $n!$ possible orders are close to
being equally likely.  A natural approach to checking empirically
whether the outcomes of a given shuffling procedure deviate from
uniformity is to apply some fixed numerical function to each of the 
permutations produced by several independent instances
of the shuffle and determine whether the resulting empirical
distribution  is close to the distribution
of the random variable that would arise from applying the chosen
function to a uniformly distributed permutation.

{\em Smoosh} shuffling (also known as {\em wash, corgi, 
chemmy} or {\em Irish} shuffling) is a simple physical mechanism
for randomizing a deck of cards -- see \cite{You11} for
an article that has a brief discussion of
smoosh shuffling and a link to a video of
the first author carrying it out, and \cite{Dia13, wiki_shuffle}
for other short descriptions.  In their forthcoming analysis
of this shuffle, \cite{BaCoDi13}
use the approach described above with
the function that takes a permutation
$\pi \in \mathfrak{S}_n$, the set of permutations of
$[n] := \{1,2, \ldots, n\}$, and returns
the cardinality of the set of labels $k \in [n-1]$ such that 
 $\pi(k+1) = \pi(k)+1$.  That is, they count the number of pairs of cards
 that were adjacent in the original deck and aren't separated
 or in a different relative order at the completion of the 
 shuffle. For example, the permutation $\pi$ of $[6]$ given by
\[
\pi = 
\begin{pmatrix}
1 & 2 & 3 & 4 & 5 & 6 & 7 \\
5 & 6 & 7 & 4 & 1 & 2 & 3
\end{pmatrix}
\]
has $\{k \in [6]: \pi(k+1) = \pi(k)+1\} = \{1,2,5,6\}$.
  
 If we write $\Pi_n$ for a random permutation that
is uniformly distributed on $\mathfrak{S}_n$ and
$\mathbf{S}_n \subseteq [n-1]$ for
the set of labels $k \in [n-1]$ such that
$\Pi_n(k+1) = \Pi_n(k) + 1$, then, in order to
support the contention that the smoosh shuffle is producing
a random permutation with a distribution close to uniform,
it is necessary to know, at least approximately, the distribution of 
the integer-valued random variable $\#\mathbf{S}_n$.
Banklader et al. \cite{BaCoDi13} use Stein's method (see, for example, 
\cite{MR2121796} for a survey), to show that the distribution of
$\#\mathbf{S}_n$ is close to a Poisson distribution with
expected value $1$ when $n$ is large.

The problem of computing $\mathbb{P}\{\#\mathbf{S}_n = 0\}$
(or, more correctly, the integer $n! \mathbb{P}\{\#\mathbf{S}_n = 0\}$)
appears  in various editions of
the 19th century textbook on combinatorics
and probability, {\em Choice and Chance} by William Allen Whitworth.
For example, Proposition XXXII in Chapter IV of \cite{Whi63}
gives 
\begin{equation}
\label{E:Whitworth}
\mathbb{P}\{\#\mathbf{S}_n = 0\}
=
\sum_{k=0}^n \frac{(-1)^k}{k!} 
+ \frac{1}{n} \sum_{k=0}^{n-1} \frac{(-1)^k}{k!}.
\end{equation}

This formula is quite suggestive.  The probability that
$\Pi_n$ has no fixed points is
$\sum_{k=0}^n \frac{(-1)^k}{k!}$ by de Montmort's \cite{deM_13}
celebrated enumeration of derangements, and
so if we write $\mathbf{T_n} \subseteq [n-1]$ for the
set of labels $k \in [n-1]$ such that $\Pi_n(k) = k$
(that is, $\mathbf{T_n}$ is the set of fixed points of $\Pi_n$
that fall in $[n-1]$), then $\mathbb{P}\{\#\mathbf{T}_n = 0\}
= \mathbb{P}\{\#\mathbf{S}_n = 0\}$ because in order for the set
$\mathbf{T}_n$ to be empty either the permutation $\Pi_n$
has no fixed points or it has $n$ as a fixed point (an event
that has probability $\frac{1}{n}$) and the resulting restriction
of $\Pi_n$ to $[n-1]$ is a permutation of $[n-1]$ that has no fixed points.

The following theorem and the remark after it were 
pointed out to us by Jim Pitman;  they show that much more is true.
Pitman's proof was similar to the enumerative one we present in
Section~\ref{S:enumerative} and he asked if there are other, more
``conceptual'' proofs.  We present two further proofs in Section~\ref{S:Markov}
and Section~\ref{S:bijective} that we hope make it clearer ``why'' the
result is true.

\begin{theorem}
\label{T:main}
For all $n \in \mathbb{N}$, the random sets $\mathbf{S}_n$ 
and $\mathbf{T}_n$ have the same distribution.
In particular, for $0 \le m \le n-1$,
\[
\begin{split}
\mathbb{P}\{\#\mathbf{S}_n = m\}
& =
\mathbb{P}\{\#\mathbf{T}_n = m\} \\
& =
\left(\frac{1}{m!} \sum_{k=0}^{n-m} \frac{(-1)^k}{k!}\right)
\frac{n-m}{n}
+
\left(\frac{1}{(m+1)!} \sum_{k=0}^{n-m-1} \frac{(-1)^k}{k!}\right)
\frac{m+1}{n}. \\
\end{split}
\]
\end{theorem}

\begin{remark}
\label{R:exchangeable}
Perhaps the most surprising consequence of this
result is that the random set $\mathbf{S}_n \subseteq [n-1]$ is exchangeable;
that is, conditional on $\# \mathbf{S}_n = m$, the conditional
distribution of $\mathbf{S}_n$ is that of $m$ random draws without
replacement from the set $[n-1]$.  This follows
because the same observation holds for the random set $\mathbf{T}_n$
by a symmetry that does not at first appear to have a counterpart
for $\mathbf{S}_n$.  For example, it does not seem obvious {\em a priori}
that $\mathbb{P}\{\{i,i+1\} \subseteq \mathbf{S}_n\}$ for some $i \in [n-2]$
should be the same as 
$\mathbb{P}\{\{j,k\} \subseteq \mathbf{S}_n\}$ for some $j,k \in [n-1]$
with $|j-k| > 1$.
\end{remark}

\begin{remark}
\label{R:count_distribution}
Once we know that $\mathbf{S}_n$ and $\mathbf{T}_n$
have the same distribution,
the formula given in Theorem~\ref{T:main} for the common distribution
distribution of $\#\mathbf{S}_n$ and $\#\mathbf{T}_n$ follows
from the well-known fact that the probability that $\Pi_n$ has
$m$ fixed points is $\frac{1}{m!} \sum_{k=0}^{n-m} \frac{(-1)^k}{k!}$
(something that follows straightforwardly from the formula above
for the probability that $\Pi_n$ has no fixed points)
coupled with the
observation that $\#\mathbf{T}_n = m$ if and only if
either $\Pi_n$ has $m$ fixed points and all of these fall in $[n-1]$
or $\Pi_n$ has $m+1$ fixed points and one of these is $n$.
\end{remark}

We present an enumerative proof of Theorem~\ref{T:main}
in Section~\ref{S:enumerative}.  Although this proof
is simple, it is not particularly illuminating.  We show in Section~\ref{S:Markov}
that the result can be derived with essentially
no computation from a comparison of two different
ways of iteratively generating uniform random permutations.  

Theorem~\ref{T:main} is, of course, equivalent to the statement
that for every subset $S \subseteq [n-1]$ the set
$\{\pi \in \mathfrak{S}_n : \{k \in [n-1] : \pi(k+1) = \pi(k)+1\} = S\}$
has the same cardinality as 
$\{\pi \in \mathfrak{S}_n : \{k \in [n-1] : \pi(k) = k\} = S\}$,
and so if the theorem holds then there must exist a bijection 
$\mathcal{H}: \mathfrak{S}_n \to \mathfrak{S}_n$ such that 
$\{k \in [n-1] : \pi(k+1) = \pi(k)+1\} 
= \{k \in [n-1] : \mathcal{H} \pi(k) = k\}$.  Conversely,
exhibiting such a bijection proves the theorem, and we present
a natural construction of one in Section~\ref{S:bijective}.

The analogue of $\mathbf{S}_n$ for circular permutations
is the random set consisting of those $k \in [n]$ such that
$\Pi(k+1 \mod n) = \Pi(k) + 1 \mod n$.  We obtain
the distribution of this random set via an enumeration
in Section~\ref{S:circular} and then present some bijective
proofs of facts suggested by the enumerative results.

Note that the latter random set is just
the set of fixed points of the commutator
$[\rho, \Pi]$, where $\rho$ is the $n$-cycle $(1,2,\ldots,n)$.  In
Section~\ref{S:commutator} we show
for a general permutation $\eta$ that, under weak conditions
on the number of fixed points and $2$-cycles of $\eta$, the 
total variation distance between the distribution of the number of fixed
points of $[\eta,\Pi]$ and a Poisson distribution with expected value $1$
is small when $n$ is large.

\begin{remark}
It is clear from Theorem~\ref{T:main} that the common distribution of
$\#\mathbf{S}_n$ and $\#\mathbf{T}_n$ is approximately Poisson with expected
value $1$ when $n$ is large.  Write 
$\mathbf{F}_n := \{k \in [n] : \Pi_n(k) = k\}$
for the set of fixed points of the uniform random permutation $\Pi_n$
and $\mathbb{Q}$ for the Poisson probability 
distribution with expected value $1$. 
It is well-known that the total variation distance between the distribution
of $\# \mathbf{F}_n$ and $\mathbb{Q}$ is amazingly small:
\[
d_{\mathrm{TV}}(\mathbb{P}\{\# \mathbf{F}_n \in \cdot\}, \mathbb{Q}) 
\le 
\frac{2^n}{n!},
\]
and so it is natural to ask whether the common distribution of
$\#\mathbf{S}_n$ and $\#\mathbf{T}_n$ is similarly close to $\mathbb{Q}$.
Because $\mathbb{P}\{\#\mathbf{T}_n \ne \#\mathbf{F}_n\} = \frac{1}{n}$,
we might suspect that the total variation distance between the distributions
of $\#\mathbf{T}_n$ and $\#\mathbf{F}_n$ is on the order of $\frac{1}{n}$,
and so the total variation distance between the distribution of 
$\#\mathbf{S}_n$ and $\mathbb{Q}$ is also of that order.  Indeed, it follows
from \eqref{E:Whitworth} that
\[
\begin{split}
\mathbb{P}\{\#\mathbf{S}_n = 0\} 
& = \mathbb{P}\{\#\mathbf{F}_n = 0\} 
+ \frac{1}{n} \sum_{k=0}^{n-1} \frac{(-1)^k}{k!} \\
& \ge
\mathbb{Q}\{0\} - \frac{2^n}{n!} 
+
\frac{1}{n} \sum_{k=0}^{n-1} \frac{(-1)^k}{k!}, \\
\end{split}
\]
and so
\[
d_{\mathrm{TV}}(\mathbb{P}\{\# \mathbf{S}_n \in \cdot\}, \mathbb{Q})
\ge
\frac{e^{-1}}{n} + \mathrm{o}\left(\frac{1}{n}\right).
\]

\end{remark}

\section{An enumerative proof}
\label{S:enumerative}
 
Our first approach to proving Theorem~\ref{T:main} is to compute
$\#\{\pi \in \mathfrak{S}_n: \pi(k_i + 1) = \pi(k_i)+1, \, 1 \le i \le m\}$
for a subset $\{k_1, \ldots, k_m\} \subseteq [n-1]$
and show that this number is 
$(n-m)! = \#\{\pi \in \mathfrak{S}_n: \pi(k_i) = k_i, \, 1 \le i \le m\}$.
This establishes that 
\[
\mathbb{P}\{\{k_1, \ldots, k_m\} \subseteq \mathbf{S}_n\}
=
\mathbb{P}\{\{k_1, \ldots, k_m\} \subseteq \mathbf{T}_n\},
\]
and an inclusion-exclusion argument completes the proof.
 
We begin by noting that we can build up a permutation of $[n]$ by first
taking the elements of $[n]$ in any order and then imagining that
we lay elements down successively so that the $h^{\mathrm{th}}$ element
goes in one of the $h$ ``slots'' defined
by the $h-1$ elements that have already been laid down, that is,
the slot before the first element, the slot after the last element,
or one of the $h-2$ slots between elements.

Consider first the set 
$\{\pi \in \mathfrak{S}_n: \pi(k+1) = \pi(k)+1\}$ for some fixed $k \in [n-1]$. 
We can count this
set by imagining that we first put down 
$k$ and $k+1$ next to each other in that order
and then successively lay down the remaining elements
$[n] \setminus \{k,k+1\}$ in such a way that
no element is ever laid down in the slot between $k$ and $k+1$.
It follows that  
\[
\#\{\pi \in \mathfrak{S}_n : \pi(k+1) = \pi(k)+1\}
=
2 \times 3 \times \cdots (n-1) = (n-1)!,
\] 
as required.  

Now consider the set 
$\{\pi \in \mathfrak{S}_n : \pi(k+1) = \pi(k)+1 \, \& \, \pi(k+2) = \pi(k+1)+1\}$ for some fixed
$k \in [n-2]$.  We can count this
set by imagining that we first put down 
$k$, $k+1$ and $k+2$ next to each other in that order
and then successively lay down the remaining elements
$[n] \setminus \{k,k+1,k+2\}$ in such a way that
no element is ever laid down in the slot between $k$ and $k+1$
or the slot between $k+1$ and $k+2$.  The number of such
permutations is thus $2 \times 3 \times \cdots \times (n-2) = (n-2)!$,
again as required.  On the other hand, suppose we fix
$k,\ell \in [n-1]$ with $|j-k| > 1$ and consider the
set 
$\{\pi \in \mathfrak{S}_n : \pi(k+1) = \pi(k)+1$ 
\, \& \, $\pi(\ell+1) = \pi(\ell)+1\}$.
We imagine that we first put down $k$ and $k+1$
next to each other in that order and then $\ell$ and $\ell+1$
next to each other in that order either before or after the
pair $k$ and $k+1$.  There are two ways to do this.  Then
we successively lay down the remaining elements 
$[n] \setminus \{k,k+1, \ell, \ell+1\}$ in such a way
that no element is ever laid down in the slot between $k$ and $k+1$
or the slot between $\ell$ and $\ell+1$.  There are
$3 \times 4 \times \cdots \times (n-2)$ ways to do this
second part of the construction, and so the number of permutations
we are considering is $2 \times 3 \times 4 \times \cdots \times (n-2)
= (n-2)!$, once again as required.

It is clear how this argument generalizes.
Suppose we have a subset $\{k_1, \ldots, k_m\} \subseteq [n-1]$
and we wish to compute  
$\#\{\pi \in \mathfrak{S_n} : \pi(k_i + 1) = \pi(k_i), \, 1 \le i \le m\}$.
We can break $\{k_1, \ldots, k_m\}$ up into $r$ ``blocks'' of consecutive
labels for some $r$.  There are $r!$ ways to lay down the blocks
and then $(r+1) \times (r+2) \times \cdots \times (n-m)$ ways
of laying down the remaining labels
$[n] \setminus \{k_1, \ldots, k_m\}$ so that no label is inserted
into a slot within one of the blocks.  Thus, the cardinality we wish to 
compute is indeed 
$r! \times (r+1) \times (r+2) \times \cdots \times (n-m) = (n-m)!$.

\section{A Markov chain proof}
\label{S:Markov}

The following proof proceeds by first showing that 
the random set $\mathbf{S}_n$ is
exchangeable and then establishing that the distribution of $\# \mathbf{S}_n$
is the same as the distribution of $\# \mathbf{T}_n$ 
without explicitly calculating either distribution.

Suppose that we build the uniform random permutations
$\Pi_1, \Pi_2, \ldots$ sequentially in the following
manner: $\Pi_{n+1}$ is obtained from $\Pi_n$ by inserting $n+1$ uniformly
at random into one of the $n+1$ ``slots'' defined by the ordered list $\Pi_n$
(i.e. as in Section~\ref{S:enumerative},
we have slots before and after the first and last elements
of the list and $n-1$ slots between successive elements).  The
choice of slot is independent of $\mathcal{F}_n$, where
$\mathcal{F}_n$ is the $\sigma$-field generated by $\Pi_1, \ldots, \Pi_n$.

It is clear that the set-valued stochastic process
$(\mathbf{S}_n)_{n \in \mathbb{N}}$ is Markovian with respect 
to the filtration $(\mathcal{F}_n)_{n \in \mathbb{N}}$.  
In fact, if we write $\mathbf{S}_n = \{X_1^n, \ldots, X_{M_n}^n\}$,
then
\[
\mathbb{P}\{\mathbf{S}_{n+1} = \{X_1^n, \ldots, X_{M_n}^n\} \setminus \{X_i^n\} 
\, | \, \mathcal{F}_n\} = \frac{1}{n+1}, \quad 1 \le i \le M_n,
\]
corresponding to $n+1$ being inserted in the slot between the successive
elements
$X_i^n$ and $X_i^n + 1$ in the list,
\[
\mathbb{P}\{\mathbf{S}_{n+1} = \{X_1^n, \ldots, X_{M_n}^n\} \cup \{n\} 
\, | \, \mathcal{F}_n\} = \frac{1}{n+1},
\]
corresponding to $n+1$ being inserted in the slot to the right of $n$,
and
\[
\mathbb{P}\{\mathbf{S}_{n+1} = 
\{X_1^n, \ldots, X_{M_n}^n\} \, | \, \mathcal{F}_n\} = \frac{(n+1) - M_n - 1}{n+1}.
\]
Moreover, it is obvious from the symmetry inherent in these 
transition probabilities and induction
that $\mathbf{S}_n$ is an exchangeable random subset of
$[n-1]$ for all $n$.
Furthermore, the nonnegative integer valued process
$(M_n)_{n \in \mathbb{N}} = (\# \mathbf{S}_n)_{n \in \mathbb{N}}$ 
is also Markovian with respect 
to the filtration $(\mathcal{F}_n)_{n \in \mathbb{N}}$
with the following transition probabilities:
\[
\mathbb{P}\{M_{n+1} = M_{n} - 1  \, | \, \mathcal{F}_n\} 
= \frac{M_{n}}{n+1},
\]
\[
\mathbb{P}\{M_{n+1} = M_{n}   \, | \, \mathcal{F}_n\} 
= \frac{(n+1) - M_n - 1}{n+1},
\]
and
\[
\mathbb{P}\{M_{n+1} = M_{n} + 1  \, | \, \mathcal{F}_n\} 
= \frac{1}{n+1}.
\]

Because the conditional distribution of $\mathbf{S}_n$
given $\# \mathbf{S}_n =m$ is, by exchangeability, the same
as that of $\mathbf{T}_n$ given $\# \mathbf{T}_n =m$ for $0 \le m \le n-1$,
it will suffice to show that the distribution of $\# \mathbf{S}_n$
is the same as that of $\# \mathbf{T}_n$.  Moreover, because
$\# \mathbf{T}_n$ has the same distribution as 
$\#\{2 \le k \le n : \Pi_n(k) = k\}$ for all $n \in \mathbb{N}$
and $\# \mathbf{S}_1 = \# \mathbf{T}_1 = 0$,
it will certainly be enough to build another sequence 
$(\Sigma_n)_{n \in \mathbb{N}}$ such that 
\begin{itemize}
\item
$\Sigma_n$ is a uniform random permutation of $[n]$ for all 
$n \in \mathbb{N}$,
\item
$(\Sigma_n)_{n \in \mathbb{N}}$ is Markovian with respect
to some filtration $(\mathcal{G}_n)_{n \in \mathbb{N}}$,
\item
$(N_n)_{n \in \mathbb{N}} 
:= (\#\{2 \le k \le n : \Sigma_n(k) = k\})_{n \in \mathbb{N}}$ 
is also Markovian with respect to the filtration 
$(\mathcal{G}_n)_{n \in \mathbb{N}}$ with the following transition
probabilities
\[
\mathbb{P}\{N_{n+1} = N_{n} - 1  \, | \, \mathcal{G}_n\} 
= \frac{N_{n}}{n+1},
\]
\[
\mathbb{P}\{N_{n+1} = N_{n}   \, | \, \mathcal{G}_n\} 
= \frac{(n+1) - N_n - 1}{n+1},
\]
and
\[
\mathbb{P}\{N_{n+1} = N_{n} + 1  \, | \, \mathcal{G}_n\} 
= \frac{1}{n+1}.
\]
\end{itemize}

We recall the simplest instance of the
{\em Chinese restaurant process} that iteratively
generates uniform random permutations (see, for example, \cite{MR2245368}).
Individuals labeled $1,2,\ldots$ successively enter a restaurant
equipped with an infinite number of round tables.  
Individual $1$ sits at some table. Suppose that after the first $n-1$
individuals have entered the restaurant we have a configuration
of individuals sitting around some number of tables.
When individual $n$ enters the restaurant he is equally likely
to sit to the immediate left of one of the individuals already
present or to sit at an unoccupied table.  The permutation
$\Sigma_n$ is defined in terms of the resulting
seating configuration by setting $\Sigma_n(i) = j$, $i \ne j$, if 
individual $j$
is sitting immediately to the left of individual $i$ and
$\Sigma_n(i) = i$ if individual $i$ is sitting by himself
at some table.  Each occupied table corresponds to a cycle
of $\Sigma_n$ and, in particular, tables with a single
occupant correspond to fixed points of $\Sigma_n$.

It is clear that if we let $\mathcal{G}_n$ be the $\sigma$-field generated by 
$\Sigma_1, \ldots, \Sigma_n$, then all of the requirements listed above for
$(\Sigma_n)_{n \in \mathbb{N}}$ and $(N_n)_{n \in \mathbb{N}}$
are met.

\section{A bijective proof}
\label{S:bijective}

As we remarked in the Introduction, in order to prove
Theorem~\ref{T:main} it suffices to find a bijection 
$\mathcal{H}: \mathfrak{S}_n \to \mathfrak{S}_n$ such that 
$\{k \in [n-1] : \pi(k+1) = \pi(k)+1\} 
= \{k \in [n-1] : \mathcal{H} \pi(k) = k\}$ for all $\pi \in \mathfrak{S}_n$. 

Not only will we find such a bijection, but we will prove an
even more general result that requires we first set up
some notation.  Fix $1 \le h < n$.  Let $\rho \in \mathfrak{S}_n$
be the permutation that maps $i \in [n]$ to $i + h \mod n \in [n]$.
Next define the following bijection of $\mathfrak{S}_n$ to itself that is
essentially the {\em transformation fondamentale} of 
\cite[Section 1.3]{MR0272642} (such bijections seem to have been first
introduced implicitly in \cite[Chapter 8]{MR0096594}).
Take a permutation $\pi$ and write it in cycle form
$(a_1, a_2, \ldots, a_r)(b_1, b_2, \ldots, b_s) \cdots 
(c_1, c_2, \ldots, c_t)$, where in each cycle the leading element 
is the least element of the cycle 
and these leading elements form a decreasing sequence. 
That is, $a_1 > b_1 > \cdots > c_1$. 
Next, remove the parentheses to form an ordered listing
$(a_1, a_2, \ldots, a_r, b_1, b_2 , \ldots , b_s, 
c_1, c_2, \ldots, c_t)$ of $[n]$ and define $\hat \pi \in \mathfrak{S}_n$
by taking $(\hat \pi(1), \hat \pi(2), \ldots, \hat \pi(n))$ to be
this ordered listing.

The following result for $h=1$ provides a bijection that establishes
Theorem~\ref{T:main}.

\begin{theorem}
\label{T:bijection}
For every $\pi \in \mathfrak{S}_n$,
\[
\{k \in [n-h] : \widehat{\rho \pi}^{-1}(k+h) = \widehat{\rho \pi}^{-1}(k) + 1\}
=
\{k \in [n-h] : \pi(k) = k\}.
\]
\end{theorem}

\begin{proof}
Suppose for some $k \in [n-h]$ that $\pi(k) = k$.  Then, $\rho \pi(k) = k+h$,
because no reduction modulo $n$ takes place.  If we write the cycle decomposition
of $\rho \pi$ in the canonical form described above, then there will be
a cycle of the form $(\ldots, k, k+h, \ldots)$ because of the convention
that each cycle begins with its least element.
After the parentheses are removed to form $\widehat{\rho \pi}$, we will have
$\widehat{\rho \pi}(j) = k$ and $\widehat{\rho \pi}(j + 1) = k+h$
for some $j \in [n]$. Hence, 
$\widehat{\rho \pi}^{-1}(k) = j$ and 
$\widehat{\rho \pi}^{-1}(k+h) = j+1 = \widehat{\rho \pi}^{-1}(k)+1$.

Conversely, suppose for some $k \in [n-h]$ that
$\widehat{\rho \pi}^{-1}(k+h) = \widehat{\rho \pi}^{-1}(k) + 1$,
so that $\widehat{\rho \pi}^{-1}(k)=j$ and 
$\widehat{\rho \pi}^{-1}(k+h) = j+1$ for some $j \in [n]$.
Then, $\widehat{\rho \pi}(j)=k$ and $\widehat{\rho \pi}(j+1) = k+h$.
The canonical cycle decomposition of $\rho \pi$
is obtained by taking the ordered listing 
$(\widehat{\rho \pi}(1), \widehat{\rho \pi}(2), \ldots, \widehat{\rho \pi}(n))$,
placing left parentheses before each element that is smaller than its
predecessors to the left, and then inserting right parentheses as necessary to produce
a legal bracketing.  It follows that $\rho \pi$ must have a cycle
of the form $(\ldots, k, k+h, \ldots)$, and hence $\rho \pi(k) = k+h$.  Thus,
$\pi(k) = k$, as required.
\end{proof}

\begin{remark}
We give the following example of the construction of $\widehat{\rho \pi}^{-1}$
from $\pi$ for the benefit of the reader.
Suppose that $n=7$ and
\[
\pi = 
\begin{pmatrix}
1 & 2 & 3 & 4 & 5 & 6 & 7 \\
7 & 2 & 6 & 4 & 1 & 3 & 5
\end{pmatrix},
\]
so that $\pi$ has canonical cycle decomposition
\[
(4) (3,6) (2) (1,7,5).
\]  
For $h=1$,
\[
\rho \pi =
\begin{pmatrix}
1 & 2 & 3 & 4 & 5 & 6 & 7 \\
1 & 3 & 7 & 5 & 2 & 4 & 6
\end{pmatrix}.
\]
The canonical cycle decomposition of $\rho \pi$ is
\[
(2,3,7,6,4,5) (1). 
\]
Thus,
\[
\widehat{\rho \pi}
=
\begin{pmatrix}
1 & 2 & 3 & 4 & 5 & 6 & 7 \\
2 & 3 & 7 & 6 & 4 & 5 & 1
\end{pmatrix}.
\]
and
\[
\widehat{\rho \pi}^{-1}
=
\begin{pmatrix}
1 & 2 & 3 & 4 & 5 & 6 & 7 \\
7 & 1 & 2 & 5 & 6 & 4 & 3
\end{pmatrix}.
\]
Note that it is indeed the case that
\[
\{k \in [6] : \pi(k) = k\} 
= 
\{2,4\} 
=
\{k \in [6] : \widehat{\rho \pi}^{-1}(k+1) = \widehat{\rho \pi}^{-1}(k) + 1\}.
\]
\end{remark}

\begin{remark}
It follows from Theorem~\ref{T:bijection} and the
argument outlined in Remark~\ref{R:count_distribution} that the
probability the random variable 
$\#\{k \in [n-h] : \Pi_n(k+h) = \Pi_n(k) + 1\}$ takes the values $m$ is
\[
\sum_{\ell = m}^{m+h}
\left(\frac{1}{\ell!} \sum_{k=0}^{n-\ell} \frac{(-1)^k}{k!}\right)
\frac{\binom{n-h}{m} \binom{h}{\ell - m}}{\binom{n}{\ell}}
\]
for $0 \le m \le n-h$.
\end{remark}

\section{Circular permutations}
\label{S:circular}

A question closely related to the ones we have
been considering so far is to ask for the distribution
of the random set 
\[
\mathbf{U}_n := \{k \in [n] : \Pi_n(k+1 \mod n) = \Pi_n(k)+1 \mod n\}.
\]
That is, we think of our deck $[n]$ as being ``circularly ordered'',
with $n$ followed by $1$, and ask for the distribution of the number
of cards that are followed immediately by their original successor when we lay
the shuffled deck out around the circumference of a circle.

\begin{proposition}
\label{P:circular}
The random set $\mathbf{U}_n$ is exchangeable with
\[
\mathbb{P}\{\#\mathbf{U}_n = m\}
=
\frac{1}{m!} 
\left(\sum_{h=0}^{n-m-1} (-1)^h \frac{1}{h!}
\frac{n}{(n-m-h)}
+ (-1)^{n-m} \frac{1}{(n-m)!} n \right)
\]
for $0 \le m \le n$.
\end{proposition}

\begin{proof}
Consider a subset $\{k_1, \ldots, k_m\} \subseteq [n]$.
We wish to compute  
\[
\#\{\pi \in \mathfrak{S}_n: 
\pi(k_i + 1 \mod n) = \pi(k_i) + 1 \mod n, \, 1 \le i \le m\}.
\]
When $m=n$ this number is clearly $n$ and when $m=0$ it is $n!$.
Consider $1 \le m \le n$.
For some positive integer $r$ we can break $\{k_1, \ldots, k_m\}$ 
up into $r$ ``runs'' of labels that are ``consecutive'' modulo
$n$; that is we can write $\{k_1, \ldots, k_m\}$  
as the disjoint union of sets $\{\ell_1, \ell_1+1, \ldots, \ell_1 + s_1-1\}$
$\{\ell_2, \ell_2 +1, \ldots, \ell_2 + s_2-1\}$, $\ldots$, 
$\{\ell_r, \ell_2 +1, \ldots, \ell_r + s_r-1\}$, where all additions are
$\mod n$ and $\ell_i + s_i \ne \ell_j$ for $i \ne j$.
This leads to $r$ disjoint ``blocks'' 
$\{\ell_1, \ell_1+1, \ldots, \ell_1 + s_1\}$, $\ldots$, 
$\{\ell_r, \ell_2 +1, \ldots, \ell_r + s_r\}$
of labels that must be kept together if we take the permutation
and join up the last element of the resulting ordered listing of $[n]$
with the first to produce a circularly ordered list.
There are $(r-1)!$ ways to circularly order the blocks.  Initially
this leaves $r$ slots between the $r$ blocks when we think of them as
being ordered around a circle.  Also, there
are initially $n-m-r$ labels that are not contained in some
block.  It follows that there are
then $r \times (r+1) \times \cdots \times (n-m-1)$ ways
of laying down the remaining $n-m-r$ elements of $[n]$
that aren't in a block so that no element is inserted
into a slot within one of the blocks.  Finally, there are $n$ places
between the $n$ circularly ordered elements 
of $[n]$ where we can cut
to produce a permutation of $[n]$. Thus, the 
cardinality we wish to compute is  
$(r-1)!
\times 
r \times (r+1) \times \cdots \times (n-m-1) \times n
= (n-m-1)! \times n$.

We see that
\[
\mathbb{P}\{\{k_1, \ldots, k_m\} \subseteq \mathbf{U}_n\}
=
\begin{cases}
1,& \quad m = 0, \\
\frac{1}{(n-m) (n-m+1) \cdots (n-1)},& \quad 1 \le m \le n-1, \\
\frac{1}{(n-1)!},& \quad m = n.
\end{cases}
\]
Consequently, by inclusion-exclusion,
\[
\begin{split}
\mathbb{P}\{\mathbf{U}_n = \{k_1, \ldots, k_m\}\}
& =
\sum_{h=0}^{n-m-1} (-1)^h \binom{n-m}{h} 
\frac{1}{(n-m-h) (n-m-h+1) \cdots (n-1)} \\
& \quad + (-1)^{n-m} \frac{1}{(n-1)!}\\
& =
\frac{(n-m)!}{(n-1)!}
\sum_{h=0}^{n-m-1} (-1)^h \frac{1}{h!}
\frac{1}{(n-m-h)} 
+ (-1)^{n-m} \frac{1}{(n-1)!}\\
\end{split}
\]

In particular, $\mathbf{U}_n$ is exchangeable and
\[
\begin{split}
\mathbb{P}\{\# \mathbf{U}_n = m\}
& =
\binom{n}{m}
\left(
\frac{(n-m)!}{(n-1)!}
\sum_{h=0}^{n-m-1} (-1)^h \frac{1}{h!}
\frac{1}{(n-m-h)} 
+ (-1)^{n-m} \frac{1}{(n-1)!}\right)\\
& =
\frac{1}{m!} 
\left(
\sum_{h=0}^{n-m-1} (-1)^h \frac{1}{h!}
\frac{n}{(n-m-h)}
+ (-1)^{n-m} \frac{1}{(n-m)!} n \right).\\
\end{split}
\]
\end{proof}

\begin{remark}
As expected, 
$\mathbb{P}\{\mathbf{U}_n = m\}$
converges to the Poisson probability
$e^{-1} \frac{1}{m!}$ as $n \to \infty$.
\end{remark}

The exchangeability of $\mathbf{U}_n$ implies
that there are is at least one bijection (and hence many) between the sets
\[
\#\{\pi \in \mathfrak{S}_n: 
\pi(k_i' + 1 \mod n) = \pi(k_i') + 1 \mod n, \, 1 \le i \le m\}
\]
and
\[
\#\{\pi \in \mathfrak{S}_n: 
\pi(k_i'' + 1 \mod n) = \pi(k_i'') + 1 \mod n, \, 1 \le i \le m\}
\]
for two subsets $\{k_1', \ldots, k_m'\}$ and $\{k_1'', \ldots, k_m''\}$
of $[n]$.   This leads to the question of
whether there is a bijection with a particularly nice description.  Rather than
pursue this question directly, we give a bijective explanation
of the following interesting consequence of 
Proposition~\ref{P:circular} from which the desired bijection
can be readily derived.

Observe that
\[
\begin{split}
\mathbb{P}\{\mathbf{U}_n = \{k_1, \ldots, k_m\}\}
& =
\sum_{h=0}^{n-m-1} (-1)^h \binom{n-m}{h} 
\frac{1}{(n-m-h) \cdots (n-1)} \\
& \quad + (-1)^{n-m} \frac{1}{(n-1)!}, \\
\end{split}
\]
whereas
\[
\begin{split}
\mathbb{P}\{\mathbf{U}_{n-m} = \emptyset\}
&  =
\sum_{h=0}^{n-m-1} (-1)^h \binom{n-m}{h} 
\frac{1}{(n-m-h) \cdots (n-m-1)} \\
& \quad + (-1)^{n-m} \frac{1}{(n-m-1)!}, \\
\end{split}
\]
so that
\begin{equation}
\label{E:circular_count}
(n-1)! \mathbb{P}\{\mathbf{U}_n = \{k_1, \ldots, k_m\}\}
=
(n-m-1)! \mathbb{P}\{\mathbf{U}_{n-m} = \emptyset\}.
\end{equation}

Let $\rho \in \mathfrak{S}_n$
be the permutation that maps $i \in [n]$ to $i + 1 \mod n \in [n]$.
Define an equivalence relation on $\mathfrak{S}_n$ by declaring that
$\pi'$ and $\pi''$ are equivalent if and only if $\rho^k \pi' = \pi''$
for some $k \in \{0,1,\ldots,n-1\}$.  We call the set of
equivalence classes the circular permutations of $[n]$
and denote this set by $\mathfrak{C}_n$.  
Note that 
$\# \mathfrak{C}_n = (n-1)!$.
We will write
$\sigma \in \mathfrak{C}_n$ as an ordered listing
$(\sigma(1), \ldots, \sigma(n))$ of $[n]$, with the understanding
that the listings produced by a cyclic permutation of the
coordinates also represent $\sigma$: a permutation 
$\pi \in \mathfrak{S}_n$ is in the equivalence class $\sigma$
if for some $k \in \{0,1,\ldots,n-1\}$ we have
$\pi(i) = \sigma(i + k \mod n)$ for $i \in [n]$.   We can 
also think of
$(\sigma(1), \ldots, \sigma(n))$ as the cycle representation of
a permutation $\tilde \sigma$
of $[n]$ consisting of a single $n$-cycle (that is,
the permutation $\tilde \sigma$
sends $\sigma(i)$ to $\sigma(i+1 \mod n)$).
Hence we can also regard  $\mathfrak{C}_n$ as the set of 
$n$-cycles in $\mathfrak{S}_n$.

If $\pi \in \mathfrak{S}_n$, then the set 
\[
\{j \in [n] : \pi^{-1}(j+1 \mod n) = \pi^{-1}(j) + 1 \mod n\}
\]
is unchanged if we replace $\pi$ by an equivalent permutation.
We denote the common value for the equivalence class 
$\sigma \in \mathfrak{C}_n$
to which $\pi$ belongs by $\Theta_n(\sigma)$.  
In terms of the $n$-cycle $\tilde \sigma \in \mathfrak{S}_n$
associated with $\sigma$,
\[
\Theta_n(\sigma) 
= 
\{j \in [n]: \tilde \sigma(j) = j + 1 \mod n\}.
\]

The identity
\eqref{E:circular_count} is equivalent to the identity
\begin{equation}
\label{E:circular_count_2}
\#\{\tau \in \mathfrak{C}_n : \Theta_n(\tau) = \{k_1, \ldots, k_m\}\}
= 
\#\{\sigma \in \mathfrak{C}_{n-m} : \Theta_n(\sigma) = \emptyset\}
\end{equation}
for any subset $\{k_1, \ldots, k_m\} \subseteq [n]$, and we will
give a bijective proof of this fact.

Consider $\sigma \in \mathfrak{C}_{n-m}$ with 
$\Theta_{n-m}(\sigma) = \emptyset$.
Suppose that we have indexed $\{k_1, \ldots, k_m\}$ so that
$k_1 < k_2 < \ldots < k_m$.  Note that $k_i \in [n-m+i]$ for $1 \le i \le m$.
We are going to recursively build circular permutations 
$\sigma = \sigma_0, \sigma_1, \ldots, \sigma_m$ with $\sigma_i \in \mathfrak{C}_{n-m+i}$
and $\Theta_{n-m}(\sigma_i) = \{k_1, k_2, \ldots, k_i\}$ for $1 \le i \le m$.

Suppose that $\sigma = \sigma_0, \ldots, \sigma_i$ have been built.
Write $\sigma_i \in \mathfrak{C}_{n-m+i}$ as 
$(h_1, \ldots, h_{n-m+i})$, where $h_1, \ldots, h_{n-m+i}$ is a listing
of  $[n-m+i]$ in some order and we recognize two such representations
as describing the same circular permutation if each can be obtained
from the other by a circular permutation of the the entries.
We first add one to each entry of $(h_1, \ldots, h_{n-m+i})$ that is
greater than $k_{i+1}$, thereby producing a vector 
that is still of length
$n-m+i$ and has entries that are a listing of
$\{1,\ldots,k_{i+1}\} \cup \{k_{i+1} + 2, \ldots, n-m+i+1\}$.
Now insert $k_{i+1} + 1$ immediately to the right of $k_{i+1}$,
thereby producing a vector that is now of length
$n-m+i+1$ and has entries that are a listing of $[n-m+i+1]$.

We can describe the procedure more formally as follows.
Either $k_{i+1} \le n-m+i$ or $k_{i+1} = n-m+i+1$.
In the first case, let $j^* \in [n-m+i]$ be such that
$\sigma_i(j^*) = k_{i+1}$ and
define $\sigma_{i+1} = (\sigma_{i+1}(1), \ldots, \sigma_{i+1}(n-m+i+1))$
by setting
\[
\sigma_{i+1}(j)
=
\begin{cases}
\sigma_i(j),& \quad \text{if $j \le j^*$ and $\sigma_i(j) \le k_{i+1}$,}\\
\sigma_i(j)+1,& \quad \text{if $j \le j^*$ and $\sigma_i(j) > k_{i+1}$,}\\
k_{i+1}+1,& \quad \text{if $j=j^*+1$,}\\
\sigma_i(j-1),& \quad \text{if $j > j^*+1$ and $\sigma_i(j) \le k_{i+1}$,}\\
\sigma_i(j-1)+1,& \quad \text{if $j > j^*+1$ and $\sigma_i(j) > k_{i+1}$.}\\
\end{cases}
\]
On the other hand, if $k_{i+1} = n-m+i+1$, then 
let $j^* \in [n-m+i]$ be such that
$\sigma_i(j^*) = 1$ and
define $\sigma_{i+1} = (\sigma_{i+1}(1), \ldots, \sigma_{i+1}(n-m+i+1))$
by setting
\[
\sigma_{i+1}(j)
=
\begin{cases}
\sigma_i(j),& \quad \text{if $j < j^*$,}\\
k_{i+1}=n-m+i+1,& \quad \text{if $j=j^*$,}\\
\sigma_i(j-1),& \quad \text{if $j > j^*$.}\\
\end{cases}
\]

It is not difficult to check in either case
that a cyclic permutation of the
coordinates in the chosen representation of $\sigma_i$ induces a cyclic
permutation in the coordinates of $\sigma_{i+1}$, and so 
$\sigma_i \mapsto \sigma_{i+1}$ is a 
well-defined map from $\mathfrak{C}_{n-m+i}$
to $\mathfrak{C}_{n-m+i+1}$.  
It is clear that $\Theta_{n-m+i+1}(\sigma_{i+1}) = \{k_1, \ldots, k_{i+1}\}$.

\begin{example}
\label{E:circular_perm}
Here are two examples of the construction just described.
Suppose that $n=7$, $\sigma = (3,1,6,5,7,2,4)$, $m=3$
and $\{k_1,k_2,k_3\} = \{3,5,6\}$.  
We begin by adding one to each entry of $\sigma$ greater than $k_1 = 3$.
This gives us
\[
(3,1,7,6,8,2,5).
\]
We then insert $4 = k_1 + 1$ immediately to the right of $k_1 = 3$ to get
\[
\sigma_1 = (\underline{3},4,1,7,6,8,2,5).
\]
Now we add one to each entry greater than $k_2 = 5$.
This gives us
\[
(3,4,1,8,7,9,2,5).
\]
We then insert $6 = k_2 + 1$ immediately to the right of $k_2 = 5$ to get
\[
\sigma_2 = (\underline{3},4,1,8,7,9,2,\underline{5},6).
\]
We next add one to each entry greater than $k_3 = 6$.
This gives us
\[
(3,4,1,9,8,10,3,5,6).
\]
Lastly, we insert $7 = k_3 + 1$ immediately to the right of $k_3 = 6$ to get
\[
\sigma_3 = (\underline{3},4,1,9,8,10,2,\underline{5},\underline{6},7).
\]
Suppose that
$n=7$, $\sigma = (6,1,3,5,4,7,2)$, $m=3$ and $\{k_1, k_2, k_3\} = \{5,8,9\}$.
Then,
\[
\sigma_1 = (7,1,3,\underline{5},6,4,8,2),
\]
\[
\sigma_2 = (7,1,3,\underline{5},6,4,\underline{8},9,2),
\]
and
\[
\sigma_3 = (7,1,3,\underline{5},6,4,\underline{8},\underline{9},10,2).
\]
\end{example}

It remains to show that each of the maps $\sigma_i \mapsto \sigma_{i+1}$
is invertible.  Suppose we have the circular permutation
$\sigma_{i+1} \in \mathfrak{C}_{n-m+i+1}$ 
with $\Theta_{n-m+i+1}(\sigma_{i+1}) = \{k_1, \ldots, k_{i+1} \}$.
The circular permutation $\sigma_i \in \mathfrak{C}_{n-m+i}$
is recovered as follows.
If $k_{i+1} \leq n-m+i$, then let $j_* \in [n-m+i+1]$ be such that
$\sigma_{i+1}(j_*) = k_{i+1}$ and
and define $\sigma_i = (\sigma_i(1), \ldots, \sigma_i(n-m+i))$
by setting
\[
\sigma_i(j)
=
\begin{cases}
\sigma_{i+1}(j),& \quad \text{if $j \le j_*$ and $\sigma_{i+1}(j) \le k_{i+1}$,}\\
\sigma_{i+1}(j)-1,& \quad \text{if $j \le j_*$ and $\sigma_{i+1}(j) > k_{i+1}+1$,}\\
\sigma_{i+1}(j+1),& \quad \text{if $j > j_*+1$ and $\sigma_{i+1}(j) \le k_{i+1}$,}\\
\sigma_{i+1}(j+1)-1,& \quad \text{if $j > j_*+1$ and $\sigma_{i+1}(j) > k_{i+1}+1$.}\\
\end{cases}
\]
On the other hand, if $k_{i+1} = n-m+i+1$, then let $j_* \in [n-m+i+1]$ 
be such that $\sigma_{i+1}(j_*) = k_{i+1} = n-m+i+1$
and define $\sigma_i = (\sigma_i(1), \ldots, \sigma_i(n-m+i))$
by setting
\[
\sigma_i(j)
=
\begin{cases}
\sigma_{i+1}(j),& \quad \text{if $j < j_*$,}\\
\sigma_{i+1}(j+1),& \quad \text{if $j \ge j_*$.}\\
\end{cases}
\]

\begin{example}
\label{E:circular_perm_inverse}
We illustrate the inversion procedure just outlined with the second
example in Example~\ref{E:circular_perm}.  We start with
\[
\sigma_3 = (7,1,3,\underline{5},6,4,\underline{8},\underline{9},10,2).
\]
remove the entry $9$ and subtract $1$ from every entry greater than $9$
to produce
\[
\sigma_2 = (7,1,3,\underline{5},6,4,\underline{8},9,2).
\]
We then remove the entry $8$ and subtract $1$ from every entry greater than $8$
to produce
\[
\sigma_1 = (7,1,3,\underline{5},6,4,8,2).
\]
Lastly, we remove the entry $5$ and subtract $1$ from every entry greater than 
$5$ to produce
\[
\sigma = (6,1,3,5,4,7,2).
\]
\end{example}

\begin{remark}
Note that
\[
\begin{split}
& \#\{\sigma \in \mathfrak{C}_n : \Theta_n(\sigma) = \emptyset\} \\
&  \quad = (n-1)! \mathbb{P}\{\mathbf{U}_n = \emptyset\} \\
& \quad \sum_{h=0}^{n-1} (-1)^h \binom{n}{h} 
(n-h-1)! \\
& \quad + (-1)^{n}. \\
\end{split}
\]
The values of this quantity for $1 \le n \le 10$
are 
\[
0, 0, 1, 1, 8, 36, 229, 1625, 13208, 120288.
\]
Recall that the
number of permutations of $[n]$ with no fixed points (that is,
the number of derangements of $n$) is given by
\[
D(n)=n!\sum_{j=0}^n \frac{(-1)^j}{j!}
\]
and values of this quantity for $1 \le n \le 10$ are
\[
0, 1, 2, 9, 44, 265, 1854, 14833, 133496, 1334961.
\]
A comparison of these sequences suggests that
\begin{equation}
\label{E:derangements_Theta}
D(n)
=
\#\{\sigma \in \mathfrak{C}_n : \Theta_n(\sigma) = \emptyset\}
+
\#\{\sigma \in \mathfrak{C}_{n+1} : \Theta_{n+1}(\sigma) = \emptyset\},
\end{equation}
and this follows readily from the observation that
\[
\begin{split}
& \binom{n+1}{h+1}(n-h-1)! - \binom{n}{h}(n-h-1)! \\
& \quad =
\frac{n!}{h! (n-h)!} (n-h-1)! \left[\frac{n+1}{h+1}-1\right] \\
& \quad =
\frac{n!}{h! (n-h)!} (n-h-1)! \frac{n-h}{h+1} \\
& \quad =
\frac{n!}{(h+1)!}. \\
\end{split}
\]

A bijective proof of \eqref{E:derangements_Theta} follows from
Corollary 2 of \cite{MR2992405}, where it is shown via a bijection
that
\begin{equation}
\label{E:derangements_quasi-fixed_points}
D(n) 
= 
\#\{\sigma \in \mathfrak{C}_{n+1} : 
\tilde \sigma(j) \ne j+1, \, j \in [n]\}.
\end{equation}
If $\sigma \in \mathfrak{C}_{n+1}$ is such that
$\tilde \sigma(j) \ne j+1$ for $j \in [n]$, then either
$\tilde \sigma(j) \ne j+1 \mod n$ for $j \in [n+1]$, so that
$\Theta_{n+1}(\sigma) = \emptyset$, or 
$\tilde \sigma(j) \ne j+1$ for $j \in [n]$ and $\tilde \sigma(n+1) = 1$.
The set of $\sigma$ in the latter category
are in a bijective correspondence with
the set of $\tau \in \mathfrak{C}_n$ such that
$\Theta_n(\tau) = \emptyset$ via the bijection that
sends a $\sigma\in \mathfrak{C}_{n+1}$ to the 
$\tau \in \mathfrak{C}_n$ given by
\[
\tilde \tau(j) 
= 
\begin{cases}
\tilde \sigma(j),& \quad \text{if $\tilde \sigma(j) \ne n+1$}, \\
\sigma(n+1) = 1,& \quad \text{if $\tilde \sigma(j) = n+1$}.\\
\end{cases}
\]

The identity \eqref{E:derangements_quasi-fixed_points} has the following
probabilistic interpretation: if $\Pi_n$ is a uniform random permutation of $[n]$
and $\Gamma_{n+1}$ is a uniform random $n+1$-cycle in $\mathfrak{S}_{n+1}$,
then
\[
\mathbb{P}\{\#\{k \in [n] : \Pi_n(k) = k\}=0\} 
= 
\mathbb{P}\{\#\{k \in [n] : \Gamma_{n+1}(k) = k+1\}=0\}.
\]
It is, in fact, the case that the two random sets
$\mathbf{F}_n := \{k \in [n] : \Pi_n(k) = k\}$ and 
$\mathbf{G}_n := \{k \in [n] : \Gamma_{n+1}(k) = k+1\}$ have the same distribution.
We will show this using an argument similar to that in Section~\ref{S:Markov}.
Suppose that $\Pi_1, \Pi_2, \ldots$ are generated using the Chinese
Restaurant Process and $\Gamma_2, \Gamma_3, \ldots$ are generated recursively 
by constructing $\Gamma_{n+1}$ from $\Gamma_n$ by picking $K$ uniformly
at random from $[n]$ and replacing $(\ldots,K,\Gamma_n(K),\ldots)$ 
in the cycle representation of $\Gamma_n$ by 
$(\ldots,K,n+1,\Gamma_n(K),\ldots)$.  It is clear that the random set
$\mathbf{F}_n$ is exchangeable.  The process $\mathbf{G}_2, \mathbf{G}_3, \ldots$
is Markovian: writing $N_n := \# \mathbf{G}_n$ and 
$\mathbf{G}_n = \{Y_1^n, \ldots, Y_{N_n}^n\}$, we have
\[
\mathbb{P}\{\mathbf{G}_{n+1} = \{Y_1^n, \ldots, Y_{N_n}^n\} \setminus \{Y_i^n\}
\, | \, \mathbf{G}_n\} 
= 
\frac{1}{n}, \quad 1 \le i \le N_n,
\]
corresponding to $n+1$ being inserted immediately to the right of $Y_i$,
\[
\mathbb{P}\{\mathbf{G}_{n+1} = \{Y_1^n, \ldots, Y_{N_n}^n\} \cup \{Y_i^n\}
\, | \, \mathbf{G}_n\} 
= 
\frac{1}{n},
\]
corresponding to $n+1$ being inserted immediately to the right of $n$,
and
\[
\mathbb{P}\{\mathbf{G}_{n+1} = \{Y_1^n, \ldots, Y_{N_n}^n\}
\, | \, \mathbf{G}_n\}
=
\frac{n - N_n - 1}{n}.
\]
It is obvious from the symmetry inherent in these transition 
probabilities and induction that $\mathbf{G}_n$
is an exchangeable random subset of $[n]$ for all
$n$. It therefore suffices to show that $N_{n+1}$ has the same
distribution as $M_n := \# \mathbf{F}_n$.  Observe that
$M_1 = N_2 = 1$.  It is clear that $N_2, N_3, \ldots$
is a Markov chain with the following transition probabilities
\[
\mathbb{P}\{N_{n+1} = N_n - 1 \, | \, N_n\} = \frac{N_n}{n},
\]
\[
\mathbb{P}\{N_{n+1} = N_n  \, | \, N_n\} = \frac{n - N_n - 1}{n},
\]
and
\[
\mathbb{P}\{N_{n+1} = N_n + 1 \, | \, N_n\} = \frac{1}{n}.
\]
It follows from the Chinese Restaurant construction that
\[
\mathbb{P}\{M_{n+1} = M_n - 1 \, | \, M_n\} = \frac{M_n}{n+1},
\]
\[
\mathbb{P}\{M_{n+1} = M_n \, | \, M_n\} = \frac{(n+1) - M_n - 1}{n+1},
\]
and
\[
\mathbb{P}\{M_{n+1} = M_n + 1 \, | \, M_n\} = \frac{1}{n+1},
\]
and so $M_n$ and $N_{n+1}$ do indeed have the same distribution
for all $n$.
\end{remark}

\section{Random commutators}
\label{S:commutator}

If we write $\rho$ for the permutation of $[n]$ given
by $\rho(i) = i+1 \mod n$, then the random set
$\mathbf{U}_n$ of Section~\ref{S:circular} is nothing other than
\[
\{i \in [n] : \rho \Pi_n(i) = \Pi_n \rho(i)\}
\]
or, equivalently, the set
\[
\{i \in [n] : \rho^{-1} \Pi_n^{-1} \rho \Pi_n(i) = i\}.
\]
This is just the set of fixed points of the
commutator $[\rho,\Pi_n] = \rho^{-1} \Pi_n^{-1} \rho \Pi_n$.
In this section we investigate  the asymptotic behavior of the
distribution of the set of fixed points of the
commutators $[\eta_n,\Pi_n]$ for a sequence of permutations 
$(\eta_n)_{n \in \mathbb{N}}$, where $\eta_n \in \mathfrak{S}_n$.

Write $\chi_n : \mathfrak{S}_n \to \{0,1,\ldots,n\}$
for the function that 
gives the number of fixed points (i.e.
$\chi_n$ is the character of the defining representation
of $\mathfrak{S}_n$).
It follows from  of \cite[Corollary 1.2]{MR1300595} (see also
 of \cite[Theorem 25]{MR2674623}) that if $\Pi_n'$ and $\Pi_n''$ are
independent uniformly distributed permutations of $[n]$, then
the distribution of $\chi_n([\Pi_n',\Pi_n''])$ is approximately Poisson
with expected value $1$ when $n$ is large. 

The results of \cite{MR1300595, MR2674623} suggest that if $n$ is
large and $\eta_n$ is a ``generic'' element of $\mathfrak{S}_n$, then
the distribution of $\chi_n([\eta_n, \Pi_n])$ should be close to Poisson
with expected value $1$.  
Of course, such a result
will not hold for arbitrary sequences $(\eta_n)_{n \in \mathbb{N}}$.  
For example, if $\eta_n$ is the identity permutation,
then $\chi_n([\eta_n,\Pi_n]) = n$.
The behavior of
$\chi_n([\eta_n,\Pi_n])$ for a deterministic sequence
$(\eta_n)_{n \in \mathbb{N}}$ 
does not appear to have been investigated in the literature.  
However, we note that
if $\tilde \Pi_n$ is an independent uniform permutation of
$[n]$, then 
\[
\begin{split}
\chi_n([\eta_n,\Pi_n]) 
& = \chi_n(\eta_n^{-1} \Pi_n^{-1} \eta_n \Pi_n) \\
& = \chi_n(\tilde \Pi_n^{-1} \, \eta_n^{-1} \Pi_n^{-1} \eta_n \Pi_n \, \tilde \Pi_n) \\
& = 
\chi_n(\tilde \Pi_n^{-1} \eta_n^{-1} \tilde \Pi_n \; \tilde \Pi_n^{-1} \, \Pi_n^{-1} \eta_n \Pi_n \, \tilde \Pi_n). \\
& = \#\{i \in [n] : U_n(i) = V_n(i)\},
\end{split}
\]
where
\[
U_n := \tilde \Pi_n^{-1} \eta_n \tilde \Pi_n
\]
and
\[
V_n := \tilde \Pi_n^{-1} \, \Pi_n^{-1} \eta_n \Pi_n \, \tilde \Pi_n
\]
are independent random permutations of $[n]$ that
are uniformly distributed on the conjugacy class of $\eta_n$.
Since $U_n$ has the same distribution as $U_n^{-1}$, we see that
$\chi_n([\eta_n,\Pi_n])$ is distributed as the number of fixed points of
the random permutation $U_n V_n$ and we could, in principle,
determine the distribution of $\chi_n([\eta_n,\Pi_n])$ if we knew the
the distribution of the conjugacy class to which $U_n V_n$ belongs.
Given a partition $\lambda \vdash n$, write $C_\lambda$ for the
conjugacy class of $\mathfrak{S}_n$ consisting of permutations with
cycle lengths given by $\lambda$ and let
$K_\lambda$ be the element $\sum_{\pi \in C_\lambda} \pi$ of the group
algebra of $\mathfrak{S}_n$.  If $C_\nu$ is another conjugacy class
with cycle lengths $\mu \vdash n$, then, writing $\ast$ for the multiplication
in the group algebra,
$K_\lambda \ast K_\mu = \sum_{\nu \vdash n} c_{\lambda \mu}^\nu K_\nu$
for nonnegative integer coefficients $c_{\lambda \mu}^\nu$.
Denote by $\gamma_n \vdash n$ the partition of $n$ given by the
cycles lengths of $\eta_n$. If we knew $c_{\gamma_n \gamma_n}^\nu$
for all $\nu \vdash n$, then we would know the
distribution of the conjugacy class to which $U_n V_n$ belongs and hence,
in principle the distribution of $\chi_n([\eta_n,\Pi_n])$.
Unfortunately, the determination of the coefficients $c_{\lambda \mu}^\nu$
appears to be a rather difficult problem.  The special case when 
$\lambda = \mu = n$ (that is, the conjugacy class of $n$-cycles
is being multiplied by itself) is treated in
\cite{MR558612, MR676430, MR1032633} and fairly explicit formulae for some
other simple cases are given in \cite{MR1165162, MR1273294}, but in general
there do not seem to be usable expressions.

In order to get a better feeling for what sort of conditions we will need
to impose on $(\eta_n)_{n \in \mathbb{N}}$ to get the hoped for Poisson
limit, we make a couple of simple observations.

Firstly, it follows that if we write
$f_n := \chi_n(\eta_n)$ for the number of fixed points of $\eta_n$, then
\[
\mathbb{E}[\chi_n([\eta_n,\Pi_n])] 
= 
n \mathbb{P}\{U_n(i) = V_n(i)\}
=
n 
\left[
\left(\frac{n - f_n}{n}\right)^2 \frac{1}{n-1}
+
\left(\frac{f_n}{n}\right)^2
\right],
\]
and so it appears that we will at least require some control on the sequence
$(f_n)_{n \in \mathbb{N}}$.

A second, and somewhat more subtle, potential difficulty 
becomes apparent if we consider
the permutation $\eta_n$ that is made up entirely of $2$-cycles (so that
$n$ is necessarily even).  In this case, 
$U_n(i) = V_n(i)$ if and only if $U_n(U_n(i)) = i = V_n(V_n(i))$,
and so $\chi_n([\eta_n,\Pi_n])$ is even.  Going a little further, we may
write $m = n/2$, take $\eta_n$ to have the cycle decomposition 
$(1,m+1)(2,m+2) \cdots (m,2m)$, and note that 
$\chi_n([\eta_n,\Pi_n]) = \#\{i \in [n] : U_n(i) = V_n(i)\}$ 
has the same distribution as 
$\#\{i \in [n] : U_n(i) = \eta_n(i)\} 
= 2 \#\{i \in [m] : U_n(i) = \eta_n(i)\}
= 2 M_n$,
where $M_n:= \sum_{i=1}^m I_{ni}$, with
$I_{ni}$ the indicator of the event $\{U_n(i) = \eta_n(i)\}$.
It is not difficult to show that
\[
\begin{split}
\mathbb{E}[M_n(M_n-1)\cdots (M_n-k+1)] 
& = 
\frac{m(m-1) \cdots (m-k+1)}{(2m-1)(2m-3) \cdots (2m - 2k + 1)} \\
& \rightarrow \frac{1}{2^k} \quad \text{as $m \to \infty$}, \\
\end{split}
\]
and so the distribution of $\chi_n([\eta_n,\Pi_n])/2$ converges to a Poisson
distribution with expected value $\frac{1}{2}$.

Returning to the case of a general permutation $\eta_n$ and
writing $t_n$ for the number of $2$-cycles in the cycle decomposition of $\eta_n$,
it seems that in order for the distribution of the random variable
$\chi_n([\eta_n,\Pi_n])$ to be close to that of a Poisson random variable
with expected value $1$ when $n$ is large we will need to at least
impose suitable conditions on $f_n$ and $t_n$. It will, in fact, suffice
to suppose that $f_n$ and $t_n$ are bounded as $n$ varies,
as the following result shows.

%
%

\begin{theorem}
Suppose that $a,b > 0$.  There exists a constant $K$ that depends
on $a$ and $b$ but not on $n \in \mathbb{N}$ such that if 
$\Pi$ is uniformly distributed on $\mathfrak{S}_n$ and
$\eta \in \mathfrak{S}_n$ has at most $a$ fixed points
and at most $b$ $2$-cycles, then the total variation
distance between the distribution of the number of fixed points
of the commutator $[\eta,\Pi]$ and a Poisson distribution with
expected value $1$ is at most $\frac{K}{n}$.
\end{theorem}

\begin{proof}
As we have observed above, the number of fixed points of 
$[\eta,\Pi]$ has the same distribution as
$\# \{i \in [n] : U(i) = V(i) \}$, where $U$ and $V$ are independent
random permutations that are uniformly distributed on the conjugacy class of 
$\eta$.  We will write $\chi$ for $\chi_n$ to simplify notation.
Similarly, we write $f$ for the number of fixed points of $\eta$
and $t$ for the number of $2$-cycles.  We assume that
$f \le a$ and $t \le b$.

Let $F_U$ and $T_U$ be the random subsets of
$[n]$ that are, respectively, the fixed points of $U$
and the elements that belong to the $2$-cycles of $U$.  Define $F_V$
and $T_V$ similarly.  Set 
\[
N := \# \{i \in [n] : U(i) = V(i), 
\; i \notin F_U \cup T_U \cup F_V \cup T_V\}.
\]
Observe that
\[
\mathbb{P}\{U(i) = V(i), 
\; i \notin F_U \cup T_U \cup F_V \cup T_V\}
=\left(\frac{n - f - 2t}{n}\right)^2 \frac{1}{n-1},
\]
so 
\[
\begin{split}
\mathbb{P}\{\chi([\eta,\Pi]) \ne N\}
& \le
\mathbb{E}[\chi([\eta,\Pi])] - \mathbb{E}[N] \\
& =
n 
\left[
\left(\frac{n - f}{n}\right)^2 \frac{1}{n-1}
+
\left(\frac{f}{n}\right)^2
- 
\left(\frac{n - f - 2t}{n}\right)^2 \frac{1}{n-1}
\right]. \\
\end{split}
\]
In particular, $n \mathbb{P}\{\chi([\eta,\Pi]) \ne N\}$
is bounded in $n$.

Let $I,J$ be chosen elements
uniformly without replacement from $[n]$
and independent of the permutations $U$ and $V$.
Set
\[
A := 
\{(I,J) \cap (F_U \cup T_U \cup F_{V} \cup T_{V}) = \emptyset\}
\]
and
\[
W := N \ind_A.
\]
Note that
\[
\mathbb{P}\{W \ne N\} \le \mathbb{P}(A^c) 
= 1 - \left(\frac{n - f - 2t}{n} \frac{n - f - 2t - 1}{n-1} \right)^2,
\]
so that $n \mathbb{P}\{W \ne N\}$, and hence 
$n \mathbb{P}\{W \ne  \chi([\eta,\Pi])\}$,
is bounded in $n$.

It will therefore suffice to show that the total variation
distance between the distribution of $W$ and a Poisson distribution with
expected value $1$ is at most a constant muliple of $\frac{1}{n}$.
We will do this using Stein's method.  More precisely, we will use
the version in \cite[Section 1]{MR2121796} that depends on the construction
of an {\em exchangeable pair}; that is, another random variable $W'$
such that $(W,W')$ has the same distribution as $(W',W)$.

Build another random permutation $V'$ by
interchanging $I$ and $J$ in the cycle representation
of $V$.  If, using a similar notation to that above, we set 
\[
N' := 
\#\{i \in [n] : U(i) = V'(i) 
\; i \notin F_U \cup T_U \cup F_{V'} \cup T_{V'}\}
\]
and
\[
W' := N' \ind_A,
\]
then $(W,W')$ is clearly an exchangeable pair.
We can represent the permutations $U$ and $V$ when the event
$A$ occurs as in Figure~\ref{fig:general_position}.
\begin{figure}[htbp]
	\centering
		\includegraphics[width=1.00\textwidth]{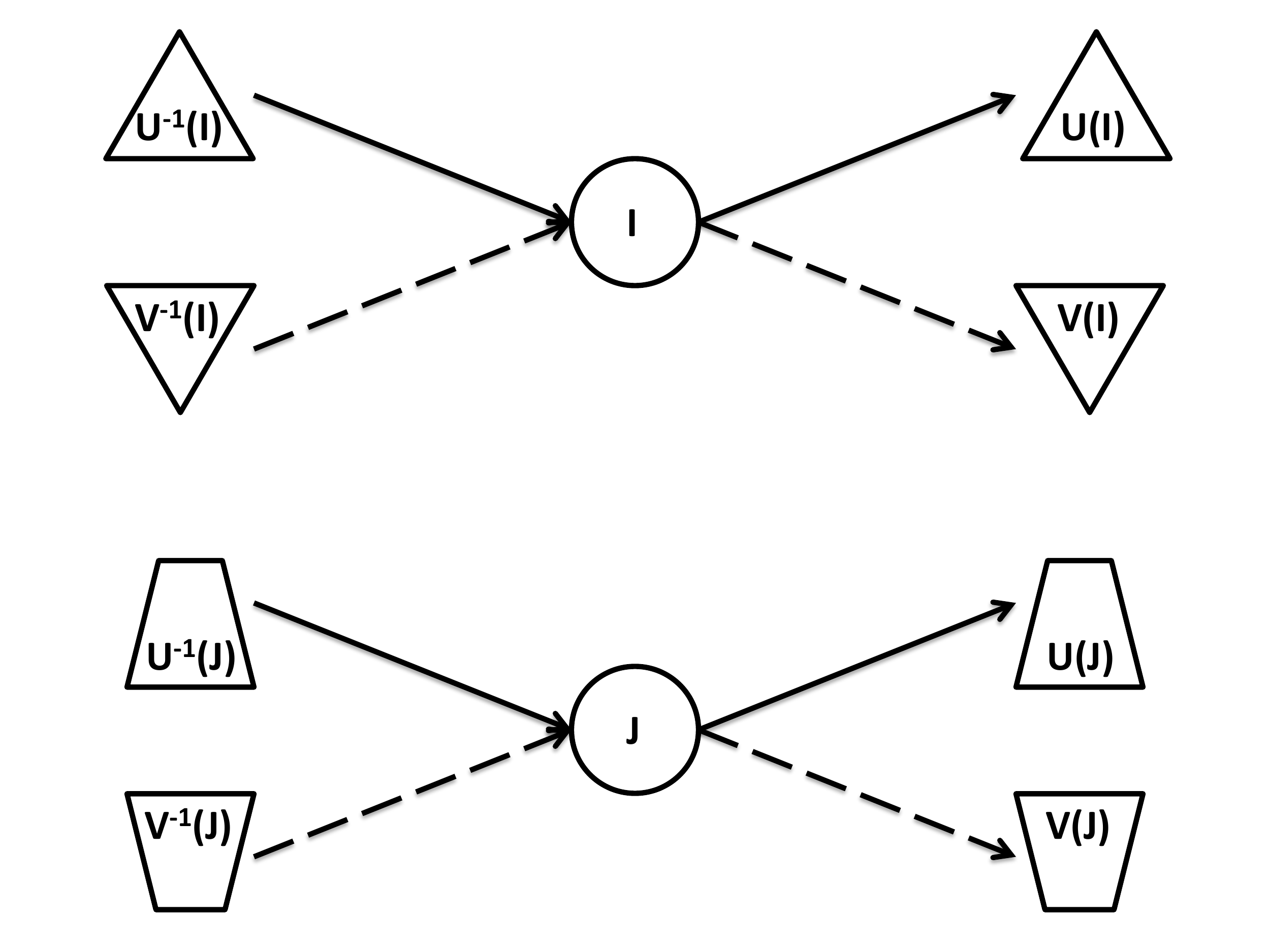}
	\caption{The effect of the permutations $U$ and $V$ on the elements 
	$I$ and $J$ when the event $A$ occurs.  
	The solid arrows depict the action of $U$ and the dashed arrows 
	depict the action of $V$. The components of the
triple $(U^{-1}(I), I, U(I))$
are distinct.  The same is true of the components of the
triples $(V^{-1}(I), I, V(I))$, $(U^{-1}(J), J, U(J))$, and
$(V^{-1}(J), J, V(J))$. However, it may happen that $U(I) = V(I)$, 
$U(I) = V(J)$, etc.}
\label{fig:general_position}
\end{figure}

We have
\begin{equation}
\label{E:diff_Wprime_W_full}
\begin{split}
W' & = W \\
& \quad - \ind\{U^{-1}(I) = V^{-1}(I)\} \cap A - \ind\{U(I) = V(I)\} \cap A \\
& \quad - \ind\{U^{-1}(J) = V^{-1}(J)\} \cap A - \ind\{U(J) = V(J)\} \cap A \\
& \quad + \ind\{U^{-1}(I) = V^{-1}(J)\} \cap A + \ind\{U(J) = V(I)\} \cap A\\
& \quad + \ind\{U^{-1}(J) = V^{-1}(I)\} \cap A + \ind\{U(I) = V(J)\} \cap A. \\
\end{split}
\end{equation}

Note that
\[
\begin{split}
& \mathbb{P}\left(\{U^{-1}(I) = V^{-1}(I)\} \cap A \, | \, (U,V)\right)
= \mathbb{P}\left(\{U(I) = V(I)\} \cap A \, | \, (U,V)\right) \\
& \quad = 
\mathbb{P}\left(\{U^{-1}(J) = V^{-1}(J)\} \cap A \, | \, (U,V)\right)
= \mathbb{P}\left(\{U(J) = V(J) \} \cap A \, | \, (U,V)\right) \\
& \quad =
\left(\frac{n - f - 2t}{n} \frac{n - f - 2t - 1}{n-1} \right)^2
\frac{W}{n-1} \\
& \quad = \frac{W}{n} + X_n,
\end{split}
\]
where $X_n$ is a random variable such that if we set
$b_n := \mathbb{E}[|X_n|]$,
then $n^2 b_n$ is bounded in $n$.
Furthermore,
\[
\begin{split}
& \mathbb{P}\left(\{U^{-1}(I) = V^{-1}(J)\} \cap A \, | \, (U,V) \right)
=
\sum_{k=1}^n 
\mathbb{P}\left(\{U^{-1}(I) = V^{-1}(J) = k\} \cap A \, | \, (U,V) \right) \\
& \quad =
\sum_{k=1}^n 
\mathbb{P}\left(\{I = U(k), \, J = V(k)\} \cap A \, | \, (U,V) \right) \\
& \quad =
n 
\left(\frac{n - f - 2t}{n} \frac{n - f - 2t - 1}{n-1} \right)^2
\left(\frac{n-1}{n} \frac{1}{n - f - 2t - 1} \right)^2 \\
& \quad =
\frac{1}{n} + c_n,
\\
\end{split}
\]
where $c_n$ is a constant such that $n^2 c_n$
is bounded in $n$,
and similar arguments show that
\[
\begin{split}
& \mathbb{P}\left(\{U(J) = V(I)\} \cap A \, | \, (U,V) \right) \\
& \quad = 
\mathbb{P}\left(\{U^{-1}(J) = V^{-1}(I)\} \cap A \, | \, (U,V) \right) 
= 
\mathbb{P}\left(\{U(I) = V(J)\} \cap A \, | \, (U,V) \right) \\
& \quad = 
n 
\left(\frac{n - f - 2t}{n} \frac{n - f - 2t - 1}{n-1} \right)^2
\left(\frac{n-1}{n} \frac{1}{n - f - 2t - 1} \right)^2 \\
& \quad = \frac{1}{n} + c_n. \\
\end{split}
\]

Suppose we can show that the probability of the intersection of
any two of the events whose indicators appear 
on the right-hand side of \eqref{E:diff_Wprime_W_full}
is at most a constant $d_n$, where $n^2 d_n$ is bounded in $n$, then
\[
\mathbb{E} 
\left[\left|W - \frac{n}{4} \mathbb{P}\{W' = W-1 \, | \, (U,V)\}\right|\right]
\le n b_n + 7 n d_n
\]

\[
\mathbb{E} \left[\left|1 - \frac{n}{4} \mathbb{P}\{W' = W+1 \, | \, (U,V)\}\right|\right]
\le n |c_n| + 7 n d_n.
\]
It will follow from the main result of \cite[Section 1]{MR2121796} that
the total variation distance between the distribution of
$W$ and a Poisson distribution with expected value $1$
is at most $\frac{C}{n}$ for a suitable constant $C$,
and hence, as we have already remarked, 
the same is true (with a larger constant) for the distribution of
$\chi([\eta,\Pi])$.

Consider the event
$\{U^{-1}(I) = V^{-1}(I)\} \cap \{U(I) = V(I)\} \cap A$, which
we represent diagrammatically in Figure~\ref{fig:A_cccccccc_intersection_1}.
\begin{figure}[htbp]
	\centering
		\includegraphics[width=1.00\textwidth]{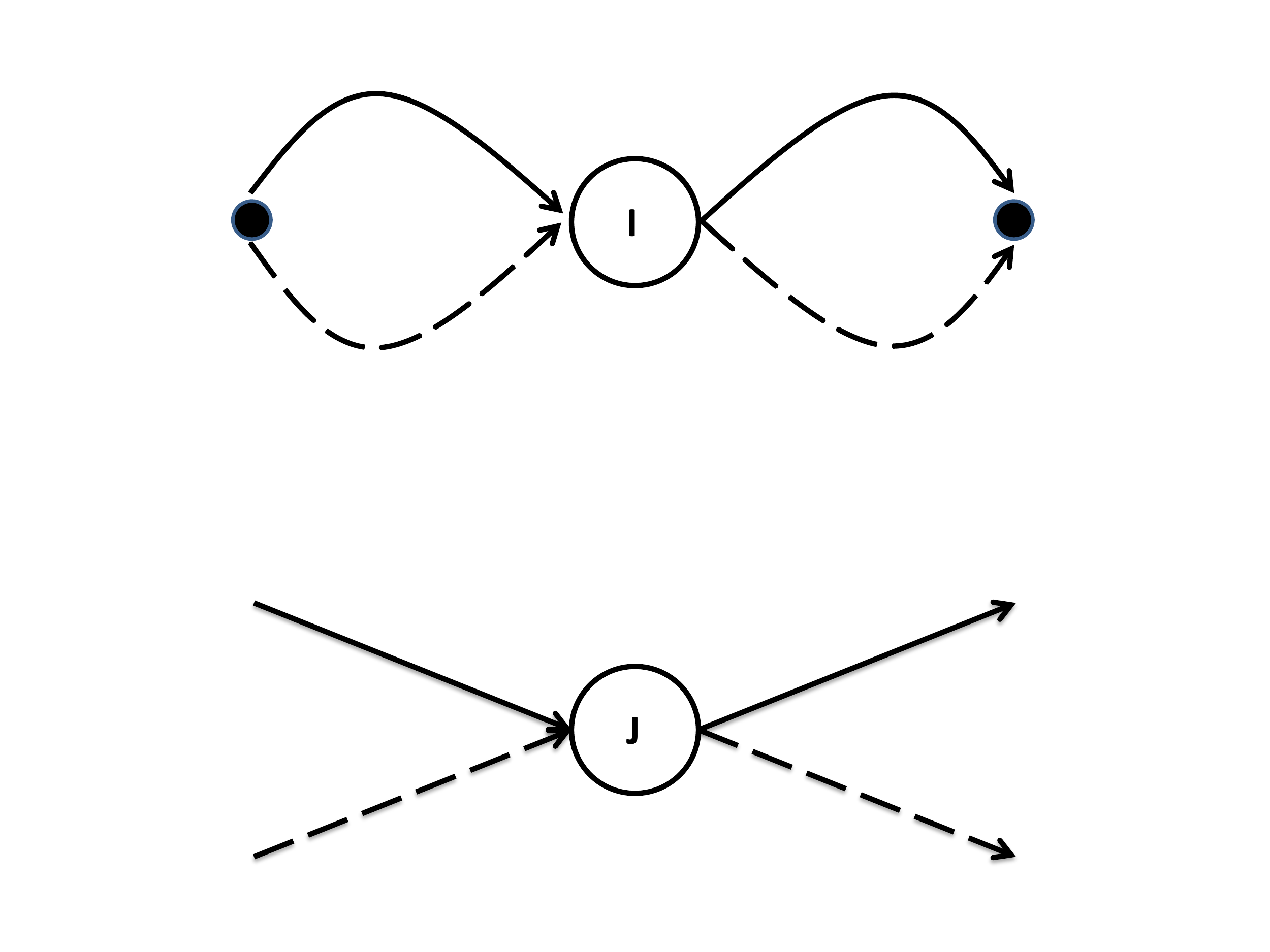}
	\caption{Diagram for the event 
	\[\{U^{-1}(I) = V^{-1}(I)\} \cap \{U(I) = V(I)\} \cap A.\]}
	\label{fig:A_cccccccc_intersection_1}
\end{figure}
The probability of this event is
\[
\left(
\frac{n - f - 2 t}{n}
\frac{n - f - 2t - 1}{n-1}
\right)^2
\frac{1}{n-2} \frac{1}{n-3}.
\]
As another example, consider the event
$\{U^{-1}(J) = V^{-1}(I)\} \cap \{U(I) = V(J)\} \cap A$,
which we represent diagrammatically in
 Figure~\ref{fig:A_cccccccc_intersection_2}.
\begin{figure}[htbp]
	\centering
		\includegraphics[width=1.00\textwidth]{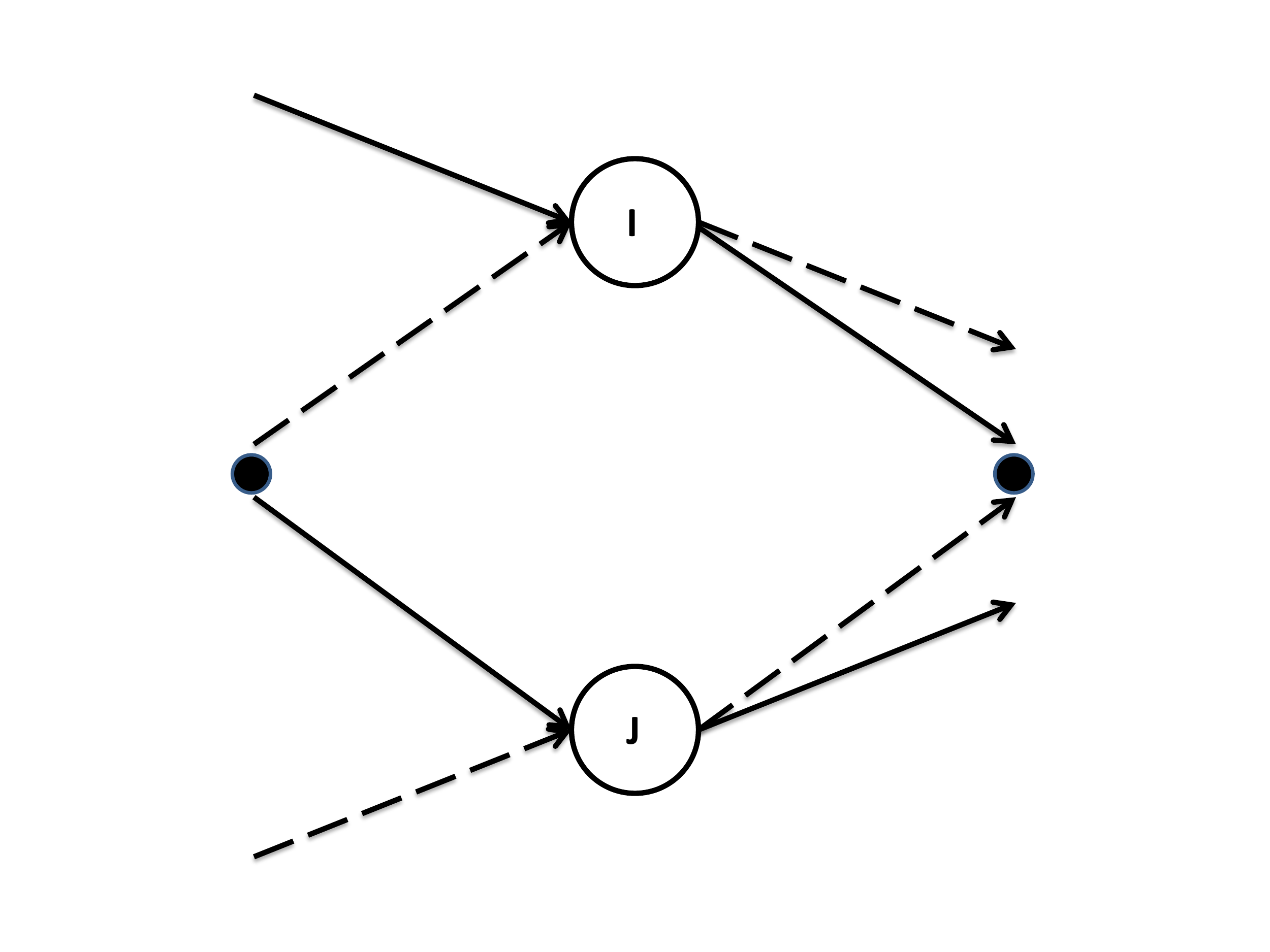}
	\caption{Diagram for the event 
	\[\{U^{-1}(J) = V^{-1}(I)\} \cap \{U(I) = V(J)\} \cap A.\]}
	\label{fig:A_cccccccc_intersection_2}
\end{figure}
The probability of this event is also
\[
\left(
\frac{n - f - 2t}{n}
\frac{n - f - 2 t - 1}{n-1}
\right)^2
\frac{1}{n-2} \frac{1}{n-3}.
\]
Continuing in this way, we see that, as required, the probability of
the intersection of any two of the events whose indicators appear 
on the right-hand side of \eqref{E:diff_Wprime_W_full}
is at most a constant $d_n$, where $n^2 d_n$ is bounded in $n$.
\end{proof}

\bigskip
\noindent
{\bf Acknowledgments:}  We thank Jim Pitman for getting us interested
in the area investigated in this paper, sharing with us
the contents of his results
Theorem~\ref{T:main} and Remark~\ref{R:exchangeable}, 
and telling us about the results in \cite{Whi63}.  
We thank Steve Butler for helpful discussions
about circular permutations and an anonymous referee for
several helpful suggestions.

\providecommand{\bysame}{\leavevmode\hbox to3em{\hrulefill}\thinspace}
\providecommand{\MR}{\relax\ifhmode\unskip\space\fi MR }
\providecommand{\MRhref}[2]{%
  \href{http://www.ams.org/mathscinet-getitem?mr=#1}{#2}
}
\providecommand{\href}[2]{#2}


\begin{thebibliography}{CDM05}

\bibitem[Bar13]{MR2992405}
Jean-Luc Baril, \emph{Statistics-preserving bijections between classical and
  cyclic permutations}, Inform. Process. Lett. \textbf{113} (2013), no.~1-2,
  17--22. \MR{2992405}

\bibitem[BCD13]{BaCoDi13}
Lauren Banklader, Marc Coram, and Persi Diaconis, \emph{Working paper on smoosh
  shuffling}, 2013, In preparation.

\bibitem[BG92]{MR1165162}
Fran{\c{c}}ois B{\'e}dard and Alain Goupil, \emph{The poset of conjugacy
  classes and decomposition of products in the symmetric group}, Canad. Math.
  Bull. \textbf{35} (1992), no.~2, 152--160. \MR{1165162 (93c:20005)}

\bibitem[Boc80]{MR558612}
G.~Boccara, \emph{Nombre de repr\'esentations d'une permutation comme produit
  de deux cycles de longueurs donn\'ees}, Discrete Math. \textbf{29} (1980),
  no.~2, 105--134. \MR{558612 (81d:05004)}

\bibitem[CDM05]{MR2121796}
Sourav Chatterjee, Persi Diaconis, and Elizabeth Meckes, \emph{Exchangeable
  pairs and {P}oisson approximation}, Probab. Surv. \textbf{2} (2005), 64--106.
  \MR{2121796 (2007b:60087)}

\bibitem[Dia13]{Dia13}
Persi Diaconis, \emph{Some things we've learned (about {M}arkov chain {M}onte
  {C}arlo)}, 2013,
  \url{http:www-stat.stanford.edu/~cgates/PERSI/papers/somethings.pdf}.

\bibitem[dM13]{deM_13}
Pierre~Raymond de~Montmort, \emph{Essay d'analyse sur les jeux de hazard},
  Jacque Quillau, Paris, 1713, Seconde Edition, {R}evue \& augment\'e de
  plusieurs lettres.

\bibitem[FS70]{MR0272642}
Dominique Foata and Marcel-P. Sch{\"u}tzenberger, \emph{Th\'eorie
  g\'eom\'etrique des polyn\^omes eul\'eriens}, Lecture Notes in Mathematics,
  Vol. 138, Springer-Verlag, Berlin, 1970. \MR{0272642 (42 \#7523)}

\bibitem[Gou94]{MR1273294}
Alain Goupil, \emph{Decomposition of certain products of conjugacy classes of
  {$S_n$}}, J. Combin. Theory Ser. A \textbf{66} (1994), no.~1, 102--117.
  \MR{1273294 (95b:20019)}

\bibitem[Gou90]{MR1032633}
\bysame, \emph{On products of conjugacy classes of the symmetric group},
  Discrete Math. \textbf{79} (1989/90), no.~1, 49--57. \MR{1032633 (90j:20010)}

\bibitem[LP10]{MR2674623}
Nati Linial and Doron Puder, \emph{Word maps and spectra of random graph
  lifts}, Random Structures Algorithms \textbf{37} (2010), no.~1, 100--135.
  \MR{2674623 (2011k:60026)}

\bibitem[Nic94]{MR1300595}
Alexandru Nica, \emph{On the number of cycles of given length of a free word in
  several random permutations}, Random Structures Algorithms \textbf{5} (1994),
  no.~5, 703--730. \MR{1300595 (95m:60017)}

\bibitem[Pit06]{MR2245368}
J.~Pitman, \emph{Combinatorial stochastic processes}, Lecture Notes in
  Mathematics, vol. 1875, Springer-Verlag, Berlin, 2006, Lectures from the 32nd
  Summer School on Probability Theory held in Saint-Flour, July 7--24, 2002,
  With a foreword by Jean Picard. \MR{2245368 (2008c:60001)}

\bibitem[Rio58]{MR0096594}
John Riordan, \emph{An introduction to combinatorial analysis}, Wiley
  Publications in Mathematical Statistics, John Wiley \& Sons Inc., New York,
  1958. \MR{0096594 (20 \#3077)}

\bibitem[Sta81]{MR676430}
Richard~P. Stanley, \emph{Factorization of permutations into {$n$}-cycles},
  Discrete Math. \textbf{37} (1981), no.~2-3, 255--262. \MR{676430 (84g:05014)}

\bibitem[Whi65]{Whi63}
William~Allen Whitworth, \emph{Choice and chance with one thousand exercises},
  Hafner Publishing Company, New York and London, 1965, Reprint of the fifth
  edition much enlarged, issued in 1901.

\bibitem[Wik13]{wiki_shuffle}
Wikipedia, \emph{Shuffling}, 2013, \url{http://en.wikipedia.org/wiki/Shuffling}
  retrieved August 25, 2013.

\bibitem[You11]{You11}
Jeffrey~R. Young, \emph{The magical mind of {P}ersi {D}iaconis}, The Chronicle
  of Higher Education (2011), October 16,
  \url{http://chronicle.com/article/The-Magical-Mind-of-Persi/129404/}.

\end{thebibliography}
\end{document}